\documentclass[12pt]{amsart}

\author{Paul \textsc{Poncet}}
\address{CMAP, \'{E}cole Polytechnique, Route de Saclay, 91128 Palaiseau Cedex, France \\
and INRIA, Saclay--\^{I}le-de-France}

\email{poncet@cmap.polytechnique.fr}

\usepackage{stmaryrd}
\usepackage{amssymb}
\usepackage{amsmath}
\usepackage{graphicx}
\usepackage[nohug]{diagrams}

\usepackage{yhmath}
\usepackage{calrsfs} 
\usepackage{times} 
\usepackage{bbm}
\usepackage{ifpdf}

\def\twoheaduparrow{\rlap{$\uparrow$}\raise.5ex\hbox{$\uparrow$}}

\newcommand{\reels}{\mathbb R}

\newtheorem{theorem}{Theorem}[section]
\newtheorem{corollary}[theorem]{Corollary}
\newtheorem{proposition}[theorem]{Proposition}
\newtheorem{lemma}[theorem]{Lemma}

\theoremstyle{definition}
\newtheorem{definition}[theorem]{Definition}
\newtheorem{example}[theorem]{Example}
\newtheorem{remark}[theorem]{Remark}
\newtheorem{problem}[theorem]{Problem}
\newenvironment{acknowledgements}[1][]{\par\vspace{0.5cm}\noindent\textbf{Acknowledgements#1.} }{\par}

\begin{document}

\title{How regular can maxitive measures be?}

\date{\today}



\begin{abstract}
We examine domain-valued maxitive measures defined on the Borel subsets of a topological space. Several characterizations of regularity of maxitive measures are proved, depending on the structure of the topological space. 
Since every regular maxitive measure is completely maxitive, this yields sufficient conditions for the existence of a cardinal density. 
We also show that every outer-continuous maxitive measure can be decomposed as the supremum of a regular maxitive measure and a maxitive measure that vanishes on compact subsets under appropriate conditions. 
\end{abstract}

\maketitle

\section{Introduction}

Maxitive measures, also known as \textit{idempotent measures}, 
are defined similarly to finitely additive measures with the supremum operation $\oplus$ in place of the addition $+$. 
In \cite[Chapter~I]{Poncet11}, we studied these measures and the related integration theory based on the  \textit{Shilkret integral}. 
We were especially interested in the idempotent analogue of the Radon--Nikodym theorem. 
In this process, we limited our considerations to maxitive measures taking values in the set of nonnegative real numbers. However, this may be quite restrictive for further applications. 

Let us have a look at classical analysis to understand why. 
In this framework, it is well known that the Radon--Nikodym theorem holds on certain classes of Banach spaces (e.g.\ reflexive spaces or separable dual spaces). To formulate such a theorem one needs to extend first the Lebesgue integral to measurable functions taking values in these spaces. This is what the Bochner integral does. 
More generally, a Banach space $B$ has the \textit{Radon--Nikodym property} if, for all measured spaces $(\mathit{\Omega},  \mathrsfs{A}, \mu)$ with finite measure $\mu$ and for all $B$-valued measures $m$ on $\mathrsfs{A}$, absolutely continuous with respect to $\mu$ and of bounded variation, there is a Bochner integrable map $f : \mathit{\Omega} \rightarrow B$ such that $m(A) = \int_{A} f \, d\mu$, for all $A \in \mathrsfs{A}$.  This property has been at the core of a great amount of research and the source of many discoveries on the structure of Banach spaces. 

One hopes to obtain analogous results in the framework of idempotent analysis. Idempotent analysis is a well established theory dating back to Zimmermann \cite{Zimmermann77} and popularized by Maslov \cite{Maslov87}; the term was coined by Kolokoltsov and made its first appearance in the papers by Kolokoltsov and Maslov \cite{Kolokoltsov89a} and \cite{Kolokoltsov89b}. 
So one must have such a powerful tool as the Bochner integral available, that would integrate $M$-valued functions, for some ``idempotent space'' $M$. One could think of $M$ e.g.\ as a complete module over the idempotent semifield $\reels^{\max}_+ = (\reels_+, \max, \times)$, but the appropriate structure still needs to be clarified. 
Jonasson \cite{Jonasson98} on the one hand, Akian \cite{Akian99} on the other hand, both worked in this direction. However, Akian chose to integrate dioid-valued (rather than module-valued) functions, and Jonasson remained in the additive paradigm. 

In order to prepare these kinds of future applications -which are not directly in the scope of this paper- we study \textit{domain}-valued maxitive measures after Akian. 
A domain is a partially ordered space with nice approximation properties. 
Well-known examples of domains are $\reels_+$, $\overline{\reels}_+$, and $[0, 1]$, which are commonly used as target sets for maxitive measures. Many attempts were made for replacing them by more general ordered structures (see Maslov \cite{Maslov87}, Greco \cite{Greco87}, Liu and Zhang \cite{Liu94}, de Cooman et al.\ \cite{deCooman01}, Kramosil \cite{Kramosil05a}). Nevertheless, the importance of supposing these ordered structures \textit{continuous} in the sense of domain theory for applications to idempotent analysis or fuzzy set theory has been identified lately. Pioneers 
were Akian \cite{Akian95, Akian99} and Heckmann and Huth \cite{Heckmann98b, Heckmann98}. See Lawson \cite{Lawson04b} for a survey on the use of domain theory in idempotent mathematics. 

In the case of Banach spaces, it must also be remarked that the Radon--Nikodym property is deeply linked with the \textit{Krein--Milman property}, which says that every bounded closed convex subset is the closed convex hull of its extreme points.  It was proved that the latter property implies the Radon--Nikodym property (see e.g.\ Benyamini and Lindenstrauss \cite[Theorem~5.13]{Benyamini00}), and the converse statement remains an open problem. 
Similar problems could be raised in the idempotent case. 

Another application we have in mind is the idempotent analogue of the Choquet integral representation theorem. In classical analysis, 
regular measures play a key role; in \cite{Poncet11} we have seen that this is also the case in the idempotent framework. This explains why we deal here with regularity properties of maxitive measures, defined on the Borel $\sigma$-algebra $\mathrsfs{B}$ of some topological space. 
On a Hausdorff space, a maxitive measure is \textit{regular} if it satisfies both following conditions for all $B \in \mathrsfs{B}$: 
\begin{itemize}
	\item \textit{inner-continuity}: 
	$$
	\nu(B) = \bigoplus_{K \in \mathrsfs{K}, K \subset B} \nu(K), 
	$$ 
	\item \textit{outer-continuity}: 
	$$
	\nu(B) = \bigwedge_{G \in \mathrsfs{G}, G \supset B} \nu(G), 
	$$ 
\end{itemize}
where $\bigoplus A$ (resp.\ $\bigwedge A$) is the supremum (resp.\ infimum) of a set $A$, and $\mathrsfs{K}$ denotes the collection of compact subsets and $\mathrsfs{G}$ that of open subsets. 
We prove a series of conditions that guarantee inner- and/or outer-continuity of maxitive measures. 
This generalizes results due to Norberg \cite{Norberg86}, Murofushi and Sugeno \cite{Murofushi93},  Vervaat \cite{Vervaat97}, O'Brien and Watson \cite{OBrien98}, Akian \cite{Akian99}, Puhalskii \cite{Puhalskii01},  Miranda et al.\ \cite{Miranda04}. 

Regularity is an important feature of maxitive measures for a different reason: a regular maxitive measure $\nu$ admits a cardinal density in the sense that, for some map $c$, we have 
$$
\nu(B) = \bigoplus_{x \in B} c(x), 
$$
for all Borel sets $B$. Numerous authors have been interested in conditions that imply the existence of such a density, hence we make the choice to revisit this problem as exhaustively as possible. 
	
For some of our proofs we follow the steps of Rie{\v{c}}anov{\'a} \cite{Riecanova84}, who focused on the regularity of certain $S$-valued set functions, for some conditionally complete ordered semigroup $S$ satisfying a series of conditions, including the separation of points by continuous functionals. We do not use directly her results, for our approach better suits the special case of domain-valued optimal measures. Indeed, a domain is not necessarily a semigroup, nor is it conditionally complete in general. 

As a last step, we prove a decomposition theorem for outer-continuous maxitive measures, that takes the following form: 
\begin{equation*}
\nu = \lfloor \nu \rfloor \oplus \bot\nu 
\end{equation*}
where $\lfloor \nu \rfloor$ is a regular maxitive measure called the \textit{regular part} of $\nu$, and $\bot\nu$ is a maxitive measure vanishing on compact subsets under appropriate conditions. 
This has the consequence that $\nu$ is regular if and only if $\bot\nu$ is zero. 

The paper is organized as follows. 
In Section~\ref{sec:domain} we recall basic domain theoretical concepts. 
In Section~\ref{sec:ts} we recall the notion of \textit{$L$-valued} maxitive measure, for some domain $L$. Then we specifically consider maxitive measures defined on the collection of Borel subsets of some topological space; we suppose that the space at stake is \textit{quasisober}, a condition that generalizes the usual Hausdorff assumption. 
We prove that regularity and tightness of maxitive measures are linked with different conditions such as existence of a cardinal density, complete maxitivity, smoothness with respect to compact saturated or closed subsets, inner-continuity. 
We focus on the case where the topological space is metrizable and the maxitive measure is optimal in Section~\ref{sec:metriz}. 
In Section~\ref{sec:decompose} we prove the announced decomposition theorem.

\section{Reminders of domain theory}\label{sec:domain}

A nonempty subset $F$ of a partially ordered set or \textit{poset} $(L,\leqslant)$ is \textit{filtered} if, for all $r, s \in F$, one can find $t \in F$ such that $t \leqslant r$ and $t \leqslant s$. A \textit{filter} of $L$ is a filtered subset $F$ such that $F = \{ s \in L : \exists r \in F, r \leqslant s \}$. 
We say that $s \in L$ is \textit{way-above} $r \in L$, written $s \gg r$, if, for every filter $F$ with an infimum $\bigwedge F$, $r \geqslant \bigwedge F$ implies $s \in F$. 
The \textit{way-above relation}, useful for studying lattice-valued upper-semicontinuous functions (see Gerritse \cite{Gerritse97} and Jonasson \cite{Jonasson98}), is dual to the usual \textit{way-below relation}, but is more appropriate in our context. Coherently, our notions of continuous posets and domains 
are dual to the traditional ones. 
We thus say that the poset $L$ is \textit{continuous} if $\twoheaduparrow r := \{ s \in L : s \gg r \}$ is a filter and $r = \bigwedge \twoheaduparrow r$, for all $r \in L$. 
Also, $L$ is \textit{filtered-complete} if every filter has an infimum. 
A \textit{domain} is then a filtered-complete continuous poset. 
In this paper, every domain considered will have a bottom element $0$. 
A poset $L$ has the \textit{interpolation property} if, for all $r, s \in L$ with $s \gg r$, there exists some $t \in L$ such that $s \gg t  \gg r$. 
In continuous posets it is well known that the interpolation property holds, see e.g.\ \cite[Theorem~I-1.9]{Gierz03}. This is a crucial feature that is behind many important results of the theory. 
For more background on domain theory, see the monograph by Gierz et al.\ \cite{Gierz03}.  

\begin{remark}
To show that an inequality $r' \geqslant r$ holds in a continuous poset $L$, it suffices to prove that, whenever $s \gg r'$, we have $s \geqslant r$. This argument will be used many times in this work. 
\end{remark}

\section{Maxitive measures on topological spaces}\label{sec:ts}

\subsection{Preliminaries on topological spaces}

Let $E$ be a topological space. We denote by $\mathrsfs{G}$ (resp.\ $\mathrsfs{F}$) the collection of open (resp.\ closed) subsets of $E$. 
The interior (resp.\ the closure) of a subset $A$ of $E$ is written $A^{o}$ (resp.\ $\overline{A}$). 
The \textit{specialization order} on $E$ is the quasiorder $\leqslant$ defined on $E$ by $x \leqslant y$ if $x \in G$ implies $y \in G$, for all open subsets $G$. 
A subset $C$ of $E$ is \textit{irreducible} if it is nonempty and, for all closed subsets $F, F'$ of $E$, $C \subset F \cup F'$ implies $C \subset F$ or $C \subset F'$. The closure of a singleton yields an irreducible closed set. 
We say that $E$ is \textit{quasisober} if every irreducible closed subset is the closure of a singleton. 
A subset $A$ of $E$ is \textit{saturated} if it is an intersection of open subsets. The \textit{saturation} of $A$, written $\uparrow\!\! A$, is the intersection of all open subsets containing $A$, and we have 
$$
\uparrow\!\! A  = \bigcap_{G \in \mathrsfs{G}, G \supset A} G = \{ x \in E : \exists a \in A, a \leqslant x \}. 
$$
If $A$ is a singleton $\{x\}$, we write $\uparrow\!\! x$ instead of $\uparrow\!\! \{x\}$. 
Note that all open subsets 
are saturated. 

We denote by $\mathrsfs{Q}$ 
the collection of (not necessarily Hausdorff) compact saturated 
subsets of $E$. 
For instance, $\uparrow\!\! x \in \mathrsfs{Q}$, for all $x \in E$. 
We shall need the following theorem, which emphasizes the role of compact saturated subsets for non-Hausdorff spaces. 

\begin{theorem}[Hofmann--Mislove]\label{thm:dual}
In a quasisober topological space, the collection $\mathrsfs{Q}$ of compact saturated subsets is closed under finite unions and filtered intersections. 
Moreover, if $(Q_j)_{j \in J}$ is a filtered family of elements of $\mathrsfs{Q}$ such that $\bigcap_{j \in J} Q_j \subset G$ for some open $G$, then $Q_j \subset G$ for some $j \in J$. 
\end{theorem}

The strong form of the Hofmann--Mislove theorem (see \cite{Hofmann81}) asserts an isomorphism between the family of compact saturated subsets of a quasisober space and the family of Scott-open filters on the lattice of open subsets of the space; Theorem~\ref{thm:dual} is then a simple corollary. Keimel and Paseka \cite{Keimel94} provided another proof, and Kov\'ar \cite{Kovar04} extended the result to generalized topological spaces. See also Jung and S\"{u}nderhauf \cite{Jung96} for an application to proximity lattices, and Norberg and Vervaat \cite{Norberg97b} for an application, in a non-Hausdorff setting, to the theory of capacities which dates back to Choquet \cite{Choquet54}.

\subsection{The Borel $\sigma$-algebra}

Let $E$ be a topological space. 
The \textit{Borel $\sigma$-algebra} of $E$ is the $\sigma$-algebra $\mathrsfs{B}$ generated by $\mathrsfs{G}$ and $\mathrsfs{Q}$; 
its elements are called the \textit{Borel} subsets of $E$. 
We also write $\mathrsfs{K}$ for the collection of compact Borel subsets of $E$. If $E$ is $T_1$ (in particular if $E$ is Hausdorff), then $\mathrsfs{K} = \mathrsfs{Q}$. 
In the case where $E$ is $T_0$, $\mathrsfs{K}$ contains all singletons $\{x\}$, for $\{x\}$ is the intersection of the compact saturated subset $\uparrow\!\! x$ with the closure $\overline{x}$ of $\{x\}$. 
In the general case ($E$ not necessarily $T_0$), we let $[x]$ denote the compact Borel subset $\uparrow\!\! x \cap \overline{x}$. This is the equivalence class of $x$ with respect to the equivalence relation $x \sim y \Leftrightarrow \overline{x} = \overline{y} \Leftrightarrow \uparrow\!\! x = \uparrow\!\! y$. Notice that $\uparrow\!\! [x] = \uparrow\!\! x$ for all $x$. 
The quotient set $E_0 = E /\!\! \sim$ equipped with the quotient topology is then a $T_0$ space, and the quotient map $\pi_0 : x \mapsto [x]$ is continuous. 

\begin{lemma}\label{lem:lemA}
Every saturated subset $A$ of $E$ satisfies $\pi_0^{-1}(\pi_0(A)) = A$, and 
every subset $A'$ of $E_0$ satisfies $\pi_0(\pi_0^{-1}(A')) = A'$. 
\end{lemma}

\begin{proof}
The second assertion is due to the surjectivity of $\pi_0$. 
To prove the first assertion, let $A$ be a saturated subset of $E$. 
It is clear that $\pi_0^{-1}(\pi_0(A)) \supset A$. To show that $\pi_0^{-1}(\pi_0(A)) \subset A$, let $x \in \pi_0^{-1}(\pi_0(A))$. Then $\pi_0(x) \in \pi_0(A)$, so there is some $a \in A$ such that $\pi_0(x) = \pi_0(a)$. In particular, $x \in \uparrow\!\! a \subset \uparrow A$. Since $A$ is saturated, we obtain $x \in A$. 
\end{proof}

\begin{lemma}\label{lem:opencompact}
The image of every open (resp.\ compact saturated) subset of $E$ under $\pi_0$ is open (resp.\ compact saturated) in $E_0$. 
The inverse image of every open (resp.\ compact saturated) subset of $E_0$ under $\pi_0$ is open (resp.\ compact saturated) in $E$. 
\end{lemma}

\begin{proof}
Let $G$ be an open subset of $E$. Since $G$ is saturated, $\pi_0^{-1}(\pi_0(G)) = G$ by Lemma~\ref{lem:lemA}. The definition of the quotient topology gives that $\pi_0(G)$ is open in $E_0$. 

Let $Q'$ be a compact saturated subset of $E_0$, and let us show that $Q := \pi_0^{-1}(Q')$ is compact saturated. So let $\mathrsfs{O}$ be a subset of $\mathrsfs{G}$ such that $Q \subset \bigcup \mathrsfs{O}$. Then $Q' = \pi_0(Q) \subset \pi_0(\bigcup \mathrsfs{O}) = \bigcup \pi_0(\mathrsfs{O})$. Since $\pi_0(O)$ is open for every $O \in \mathrsfs{O}$ and $Q'$ is compact, we get $Q' \subset \pi_0(O_1) \cup \cdots \cup \pi_0(O_k)$ for some $O_1, \ldots, O_k \in \mathrsfs{O}$, so that $Q \subset \pi_0^{-1}(\pi_0(O_1) \cup \cdots \cup \pi_0(O_k)) = O_1 \cup \cdots \cup O_k$. This shows that $Q$ is compact; the proof that $Q$ is saturated is not difficult and left to the reader. 

The remaining assertions directly follow from the continuity of $\pi_0$. 
\end{proof}

\begin{lemma}\label{lem:lemB}
For all Borel subsets $B$ of $E$, the set $\pi_0(B)$ is a Borel subset of $E_0$ and satisfies $\pi_0^{-1}(\pi_0(B)) = B$. 
For all Borel subsets $B'$ of $E_0$, the set $\pi_0^{-1}(B')$ is a Borel subset of $E$ and satisfies $\pi_0(\pi_0^{-1}(B')) = B'$. 
\end{lemma}

\begin{proof}
For the first assertion, let $\mathrsfs{A}$ be the collection of all $B \in \mathrsfs{B}$ such that $\pi_0(B)$ is a Borel subset of $E_0$ and $\pi_0^{-1}(\pi_0(B)) = B$. 
It is easily seen that $\mathrsfs{A}$ is a $\sigma$-algebra. Moreover, the combination of Lemma~\ref{lem:lemA} and Lemma~\ref{lem:opencompact} implies that 
$\mathrsfs{A}$ contains both $\mathrsfs{G}$ and $\mathrsfs{Q}$. As a consequence, $\mathrsfs{A} = \mathrsfs{B}$. 

With the help of Lemma~\ref{lem:lemA}, the second assertion of the lemma can be proved similarly. 
\end{proof}

In the following result, the concept of \textit{reflection} refers to category theory. 

\begin{theorem}
Every topological space $E$ has a $T_0$-reflection $(E_0, \pi_0)$ given by $\pi_0 : E \rightarrow E_0, x \mapsto [x]$, 
where $E_0 = E/\!\! \sim$ is the quotient set equipped with the quotient topology, which is a $T_0$ topology, and $\sim$ is the equivalence relation $x \sim y \Leftrightarrow \overline{x} = \overline{y}$. 
Moreover, the correspondence $B' \mapsto \pi_0^{-1}(B')$ is an isomorphism of Borel $\sigma$-algebras. 
\end{theorem}

\begin{proof}
We have to show first that, for all $T_0$ spaces $X$ and all continuous maps $f : E \rightarrow X$, there exists a unique continuous map $f_0 : E_0 \rightarrow X$ such that the following diagram commutes:  
\begin{diagram}
E & \rTo_{\pi_0}  & E_0  \\
  & \rdTo_{f} &  \dDashto<{f_0}  \\
  &               &  X  \\
\end{diagram}
But $f_0$ can be explicitly defined by $f_0([x]) = f(x)$, for if $[x] = [y]$ then $f(x) = f(y)$. Indeed, let $G'$ be an open subset containing $f(y)$. Then $y \in f^{-1}(G')$. Since $f$ is continuous, $f^{-1}(G')$ is open in $E$. Considering that $x \geqslant y$, this implies that $x$ is also in  $f^{-1}(G')$, thus $f(x) \in G'$. If one inverts the roles of $x$ and $y$, we deduce that $f(x) \sim f(y)$. But $X$ is $T_0$, so that $f(x) = f(y)$. The uniqueness of $f_0$ then directly follows.  

To conclude the proof, firstly recall that, by Lemma~\ref{lem:lemB}, $\pi_0^{-1}(B') \in \mathrsfs{B}$ for all Borel subsets $B'$ of $E_0$. 
Secondly, the map $\Psi : B' \mapsto \pi_0^{-1}(B')$ is both surjective and injective thanks to Lemma~\ref{lem:lemB}, and we easily deduce that it is an isomorphism of Borel $\sigma$-algebras. 
\end{proof}

In this paper, the maxitive measures considered will only be defined on the Borel $\sigma$-algebra of the topological space $E$ at stake. By the previous theorem, we thus may assume that $E$ be $T_0$ without loss of generality.
However, we believe it interesting, from a formal point of view, to explicitly work in a non-$T_0$ setting. So we make the choice to keep on with general (not necessarily $T_0$) topological spaces; for that reason the following result will be useful. 
 
\begin{corollary}\label{lem:tilde}
For all Borel subsets $B$ of $E$, $x \in B$ implies $[x] \subset B$. 
\end{corollary}

\begin{proof}
Let $x \in B$, and let $y \in [x]$. We want to show that $y \in B$. But $y \in [x]$ implies $\pi_0(y) = \pi_0(x) \in \pi_0(B)$. Thus, $y \in \pi_0^{-1}(\pi_0(B)) = B$ by Lemma~\ref{lem:lemB}. 
\end{proof}


%


\subsection{Regular maxitive measures} 

Let $E$ be a topological space with Borel $\sigma$-algebra $\mathrsfs{B}$, and let $L$ be a filtered-complete poset with a bottom element, that we denote by $0$. 
An $L$-valued \textit{maxitive measure} (resp.\ \textit{$\sigma$-maxitive measure}, \textit{completely maxitive measure}) on $\mathrsfs{B}$ is a map $\nu : \mathrsfs{B} \rightarrow L$ such that $\nu(\emptyset) = 0$ and, for every finite (resp.\ countable, arbitrary) family $\{B_j\}_{j\in J}$ of elements of $\mathrsfs{B}$ such that $\bigcup_{j \in J} B_j \in \mathrsfs{B}$, the supremum of $\{ \nu(B_j) : j\in J \}$ exists and satisfies 
$$
\nu(\bigcup_{j\in J} B_j) = \bigoplus_{j \in J} \nu(B_j). 
$$
Note that this definition implies that the image $\nu(\mathrsfs{B})$ is a sup-subsemilattice of $L$ containing $0$ (even though $L$ itself need not be a sup-semilattice), and the corestriction of $\nu$ to $\nu(\mathrsfs{B})$ is a sup-semilattice morphism. 

An $L$-valued maxitive measure $\nu$ on $\mathrsfs{B}$ is \textit{regular} if it satisfies both following relations for all $B \in \mathrsfs{B}$: 
\begin{itemize}
	\item \textit{inner-continuity}: 
	$$
	\nu(B) = \bigoplus_{K \in \mathrsfs{K}, K \subset B} \nu(\uparrow\!\! K), 
	$$ 
	\item \textit{outer-continuity}: 
	$$
	\nu(B) = \bigwedge_{G \in \mathrsfs{G}, G \supset B} \nu(G).  
	$$ 
\end{itemize}

\begin{example}\label{ex:nuplus}
Assume that $L$ is a domain. 
For an $L$-valued ($\sigma$-)maxitive measure $\nu$ on $\mathrsfs{G}$, the set $\{  \nu(G) : G \in \mathrsfs{G}, G \supset B \}$ is filtered for all $B \in \mathrsfs{B}$, so one can define a map $\nu^{+}$ on $\mathrsfs{B}$ by 
$$
\nu^{+}(B) = \bigwedge_{G \in \mathrsfs{G}, G \supset B} \nu(G). 
$$
Then, by \cite[Corollary~2.4]{Poncet10}, $\nu^{+}$ is an outer-continuous ($\sigma$-)maxitive measure (see also Akian \cite[Proposition~3.1]{Akian99}). Moreover, $\nu^{+}$ is inner-continuous (hence regular) if $\nu$ is inner-continuous (combine Lemma~\ref{lem:locconv} and Proposition~\ref{lem:loccomp} below). 
\end{example}

We shall also use weakened notions of inner- and outer-continuity for an $L$-valued maxitive measure $\nu$ on $\mathrsfs{B}$: 
\begin{itemize}
	\item \textit{weak inner-continuity}: 
	$$
	\nu(G) = \bigoplus_{K \in \mathrsfs{K}, K \subset G} \nu^{+}(K), \quad\quad \mbox{ for all $G \in \mathrsfs{G}$ }, 
	$$ 
	\item \textit{weak outer-continuity}: 
	$$
	\nu(K) = \bigwedge_{G \in \mathrsfs{G}, G \supset K} \nu(G), \quad\quad \mbox{ for all $K \in \mathrsfs{K}$ }.    
	$$ 
\end{itemize}

The following result ensures that the terminology we use is consistent. 

\begin{lemma}\label{lem:locconv}
An inner- (resp.\ outer-)continuous maxitive measure on $\mathrsfs{B}$ is weakly inner- (resp.\ weakly outer-)continuous. 
\end{lemma}

\begin{proof}
The easy proof is left to the reader. 
\end{proof}

The notion of weak inner-continuity can be characterized as follows. 

\begin{lemma}\label{lem:wic}
Assume that $L$ is a domain. Let $\nu$ be an $L$-valued maxitive measure on $\mathrsfs{B}$. 
Then $\nu$ is weakly inner-continuous if and only if 
\begin{equation}\label{eq:eqo}
\nu(\bigcup \mathrsfs{O}) = \bigoplus \nu(\mathrsfs{O}), 
\end{equation}
for all families $\mathrsfs{O}$ of open subsets of $E$. 
\end{lemma}

\begin{proof}
First we suppose that $\nu$ is weakly inner-continuous. Let $\mathrsfs{O}$ be a family of open subsets of $E$, and let $G = \bigcup \mathrsfs{O}$. 
The identity we need to show will be satisfied if we prove that $\nu^{+}(K) \leqslant \bigoplus \nu(\mathrsfs{O})$ for all compact Borel subsets $K \subset G$. 
But for such a $K$, there are open subsets $O_1, \ldots, O_n$ in $\mathrsfs{O}$ such that $K \subset O_1 \cup \ldots \cup O_n$, so that $\nu^{+}(K) \leqslant \nu(O_1) \oplus \ldots \oplus \nu(O_n) \leqslant \bigoplus \nu(\mathrsfs{O})$. 

Conversely, suppose that Equation~(\ref{eq:eqo}) holds for all families $\mathrsfs{O}$ of open subsets of $E$. 
To prove that $\nu$ is weakly inner-continuous, fix some $G \in \mathrsfs{G}$, let $u$ be an upper-bound of $\{ \nu^{+}(K) : K \in \mathrsfs{K}, K \subset G \}$, and let $s \gg u$. 
Then for all $x \in G$, $s \gg \nu^{+}([x])$, so there is some $G_x \in \mathrsfs{G}$ such that $G_x \ni x$ and $s \geqslant \nu(G_x)$. By Equation~(\ref{eq:eqo}) we have $s \geqslant \nu(\bigcup_{x \in G} G_x) \geqslant \nu(G)$. Since $L$ is continuous, we get $u \geqslant \nu(G)$, and the result follows. 
\end{proof}

The following lemma characterizes weak outer-continuity. 

\begin{lemma}\label{lem:reg0}
Assume that $L$ is a domain. 
Let $\nu$ be an $L$-valued maxitive measure on $\mathrsfs{B}$. Then 
$$
\nu^{+}(K) = \bigoplus_{x \in K} \nu^{+}([x]), 
$$
for all $K \in \mathrsfs{K}$. 
As a consequence, $\nu$ is weakly outer-continuous if and only if $\nu([x]) = \nu^{+}([x])$ for all $x \in E$.  
\end{lemma}

\begin{proof}
We let $c^{+} : x \mapsto \nu^{+}([x])$. 
Let $u \in L$ be an upper-bound of $\{ c^{+}(x) : x \in K \}$ and let $s \gg u$. Then, for each $x \in K$, $s \gg c^{+}(x)$, so there is some open subset $G_x \ni x$ such that $s \geqslant \nu(G_x)$. Since $K$ is compact and $\bigcup_{x \in K} G_x \supset K$, we can extract a finite subcover and write $\bigcup_{j = 1}^k G_{x_j} \supset K$. Thus, $s \geqslant \nu^{+}(K)$. Since $L$ is continuous, this implies that $u \geqslant \nu^{+}(K)$, so that $\nu^{+}(K)$ is the least upper-bound of $\{ c^{+}(x) : x \in K \}$. 

Now we prove the announced equivalence. First assume that $\nu([x]) = \nu^{+}([x])$ for all $x \in E$. Then, for every compact Borel subset $K$, 
$$
\nu^{+}(K) = \bigoplus_{x \in K} \nu^{+}([x]) = \bigoplus_{x \in K} \nu([x]) \leqslant \nu(K) \leqslant \nu^{+}(K), 
$$
so $\nu^{+}(K) = \nu(K)$, i.e.\ $\nu$ is weakly outer-continuous. 
Conversely, if $\nu$ is weakly outer-continuous then, for all $x \in E$, $[x]$ is a compact Borel subset, hence $\nu([x]) = \nu^{+}([x])$.  
\end{proof}

It happens that we recover regularity if we combine weak inner- and weak outer-continuity. 

\begin{proposition}\label{lem:loccomp}
Assume that $L$ is a domain. 
Then every $L$-valued maxitive measure on $\mathrsfs{B}$ that is both weakly outer-continuous and weakly inner-continuous is regular and completely maxitive. 
\end{proposition}

\begin{proof}
Let $\nu$ be an $L$-valued weakly outer-continuous and weakly inner-continuous maxitive measure. 
Assume that, for some $B \in \mathrsfs{B}$, $\nu^{+}(B)$ is not the least upper-bound of $\{ \nu(K) : K \in \mathrsfs{K}, K \subset B \}$. Then there exists some upper-bound $u \in L$ of $\{ \nu(K) : K \in \mathrsfs{K}, K \subset B \}$ such that $u \not\geqslant \nu^{+}(B)$. 
Since $L$ is continuous, there exists some $s \gg u$ with $s \not\geqslant \nu^{+}(B)$. 
If $x \in B$, then $K_x = [x]$ is a compact Borel subset, and $K_x \subset B$ by Lemma~\ref{lem:tilde}. 
So $s \gg \nu(K_x) = \nu^{+}(K_x)$ since $\nu$ is weakly outer-continuous, hence there exists some $G_x \ni x$ such that $s \geqslant \nu(G_x)$. Since  $\nu$ is weakly inner-continuous, we deduce $s \geqslant \nu(G)$, where $G = \bigcup_{x \in B} G_x \supset B$, so that $s \geqslant \nu^{+}(B)$, a contradiction. 

So we have proved that $\nu^{+}(B) = \bigoplus \{ \nu(K) : K \in \mathrsfs{K}, K \subset B \}$, for all $B \in \mathrsfs{B}$. From this we deduce that $\nu^{+}(B) = \nu(B)$, i.e.\ $\nu$ is outer-continuous. This implies that $\nu(\uparrow\!\! K) = \nu^{+}(\uparrow\!\! K) = \nu^{+}(K) = \nu(K)$ for all $K \in \mathrsfs{K}$, and now inner-continuity of $\nu$ is clear. 

To prove that $\nu$ is completely maxitive, we let $(B_j)_{j\in J}$ be some family of Borel subsets such that $B := \bigcup_{j\in J} B_j \in \mathrsfs{B}$. 
We also take an upper-bound $u$ of $\{ \nu(B_j) : j \in J \}$ and some $s \gg u$. 
Since $\nu$ is outer-continuous there exists, for all $j \in J$, some $G_j \supset B_j$ such that $s \geqslant \nu(G_j)$. By Equation~(\ref{eq:eqo}) in Lemma~\ref{lem:wic} we get $s \geqslant \nu(\bigcup_{j\in J} G_j)$, so that $s \geqslant \nu(B)$. Since $L$ is continuous we obtain $u \geqslant \nu(B)$. As a consequence, $\nu(B)$ is the least upper-bound of $\{ \nu(B_j) : j \in J \}$. This proves that $\nu$ is completely maxitive. 
\end{proof}	

The following result improves \cite[Corollary~3.12]{Akian99}. 

\begin{corollary}\label{coro:sc}
Assume that $L$ is a domain. 
Then, on a second-countable topological space, every $L$-valued weakly outer-continuous $\sigma$-maxitive measure is regular. 
\end{corollary}

\begin{proof}
Let $E$ be second-countable and $\nu$ be an $L$-valued weakly outer-continuous $\sigma$-maxitive measure on $\mathrsfs{B}$. 
Since $E$ is second-countable, there is some countable base $\mathrsfs{U}$ for the topology $\mathrsfs{G}$. 
To prove that $\nu$ is regular, we want to use Proposition~\ref{lem:loccomp}, thus we show that $\nu$ is weakly inner-continuous. So let $\mathrsfs{O}$ be a family of open subsets of $E$, and let $G = \bigcup \mathrsfs{O}$. 
We let $\mathrsfs{V} = \{ V \in \mathrsfs{U} : \exists O \in \mathrsfs{O}, V \subset O \}$. 
Since $\mathrsfs{V}$ is countable with union $G$ and $\nu$ is $\sigma$-maxitive, we deduce that $\nu(G) = \bigoplus \nu(\mathrsfs{V}) \leqslant \bigoplus \nu(\mathrsfs{O})$. By Lemma~\ref{lem:wic}, $\nu$ is weakly inner-continuous, and the proof is complete.
\end{proof}

\subsection{Smoothness}

From now on, all ($L$-valued) maxitive measures are assumed to be defined on the Borel $\sigma$-algebra $\mathrsfs{B}$ of a topological space $E$.  
If $\mathrsfs{A}$ is a collection of elements of $\mathrsfs{B}$ closed under filtered intersections, the maxitive measure $\nu$ is \textit{$\mathrsfs{A}$-smooth} if  
\begin{equation}\label{eq:filta}
\bigwedge_{j\in J} \nu(A_j) = \nu(\bigcap_{j \in J} A_j), 
\end{equation}
for every filtered family $(A_j)_{j\in J}$ of elements of $\mathrsfs{A}$. 

An $L$-valued maxitive measure $\nu$ on $\mathrsfs{B}$ is called \textit{saturated} if for all $K \in \mathrsfs{K}$ we have $\nu(K) = \nu(\uparrow\!\! K)$. Inner-continuous maxitive measures and weakly outer-continuous maxitive measures are always saturated, while weak inner-con\-tinuity does not imply saturation in general. Note however that saturation is always satisfied if the space $E$ is $T_1$. 

Variants of Propositions~\ref{prop:k} and \ref{prop:f} below were formulated and pro\-ved in \cite{Akian95} in the case where $E$ is a Hausdorff topological space and $L$ is a continuous lattice, see also \cite[Proposition~13]{Heckmann98}. 
Another variant of the following result is \cite[Proposition~2.2(a)]{Norberg97b}, which treats the case of real-valued capacities on non-Hausdorff spaces. 

\begin{proposition}\label{prop:k}
Assume that $L$ is a domain. 
Then, on a quasisober space, 
every $L$-valued weakly outer-continuous maxitive measure   
is $\mathrsfs{Q}$-smooth saturated. 
The converse statement holds in locally compact quasisober spaces. 
\end{proposition}

\begin{proof}
Let $E$ be quasisober, let $\nu$ be an $L$-valued weakly outer-continuous maxitive measure on $\mathrsfs{B}$, and let $(Q_j)_{j\in J}$ be a filtered family of compact saturated subsets of $E$.  
Recall that $Q = \bigcap_{j \in J} Q_j$ is compact saturated, since $E$ is assumed quasisober. 
The set $\{  \nu(Q_j) : j \in J \}$ admits $\nu(Q)$ as a lower-bound. Take another lower-bound $\ell$, and let $G \in \mathrsfs{G}$ such that $G \supset Q$. 
By the Hofmann--Mislove theorem (Theorem~\ref{thm:dual}), there is some $j_0 \in J$ such that $G \supset Q_{j_0}$. Thus, $\nu(G) \geqslant \nu(Q_{j_0})$, so that $\nu(G) \geqslant \ell$, for all $G \supset Q$. Since $\nu$ is weakly outer-continuous, we deduce that $\nu(Q) \geqslant \ell$. 
We have shown that $\nu(Q)$ is the infimum of $\{ \nu(Q_j) : j \in J \}$. This proves that $\nu$ is $\mathrsfs{Q}$-smooth. 

Now assume that $E$ is locally compact quasisober, and let $\nu$ be an $L$-valued $\mathrsfs{Q}$-smooth saturated maxitive measure on $\mathrsfs{B}$. 
If $Q$ is a compact saturated subset, then by local compactness of $E$ there exists a filtered family $(Q_j)_{j \in J}$ of compact saturated subsets with $\bigcap_{j\in J} Q_j = Q$ and $Q \subset Q_j^{o}$. Since $\nu$ is $\mathrsfs{Q}$-smooth, this implies that 
$$
\nu(Q) = \bigwedge_{G \in \mathrsfs{G}, G \supset Q} \nu(G), 
$$
i.e.\ $\nu(Q) = \nu^{+}(Q)$, for all $Q \in \mathrsfs{Q}$. 
Let us show that $\nu$ and $\nu^{+}$ coincide on $\mathrsfs{K}$. If $K \in \mathrsfs{K}$, then $\nu(K) = \nu(\uparrow\!\! K)$ since $\nu$ is saturated. Also, because $G \supset \uparrow\!\! K$ if and only if $G \supset K$ for all open subsets $G$, we have $\nu^{+}(\uparrow\!\! K) = \nu^{+}(K)$. So this gives $\nu(K) = \nu(\uparrow\!\! K) = \nu^{+}(\uparrow\!\! K) = \nu^{+}(K)$, and we have shown that $\nu$ is weakly outer-continuous. 
\end{proof}

\begin{remark}
The first part of Proposition~\ref{prop:k} remains true for $L$-valued weakly outer-con\-tinuous monotone set functions. 
\end{remark}

\subsection{Tightness}

Tightness of maxitive measures can be defined by analogy with tightness of additive measures, so we say that an $L$-valued maxitive measure $\nu$ on $\mathrsfs{B}$ is \textit{tight} if 
$$
\bigwedge_{K \in \mathrsfs{K}} \nu(E \setminus K) = 0.  
$$ 
The following lemma slightly extends \cite[Theorem~III-2.11]{Gierz03}, which states that every continuous sup-semilattice is join-continuous. 

\begin{lemma}\label{lem:jcont}
Assume that $L$ is a domain. Let $F$ be a filter of $L$ and $t \in L$ such that, for all $f \in F$, $t \oplus f$ exists. Then $t \oplus \bigwedge F$ exists and satisfies 
$
t \oplus \bigwedge F = \bigwedge (t \oplus F). 
$
\end{lemma}

\begin{proof}
The subset $t \oplus F$ is filtered, hence has an infimum. 
Suppose that $\bigwedge (t \oplus F)$ is not the least upper-bound of $\{t, \bigwedge F \}$. Then there exists some upper-bound $u$ of $\{t, \bigwedge F \}$ such that $u \not\geqslant \bigwedge (t \oplus F)$. Since $L$ is continuous, there is some $s \gg u$ such that $s \not\geqslant \bigwedge (t \oplus F)$. Remembering that $u \geqslant \bigwedge F$, there is some $f \in F$ such that $s \geqslant f$. Also, $s \geqslant u \geqslant t$, so that $s \geqslant t \oplus f \geqslant \bigwedge (t \oplus F)$, a contradiction. 
\end{proof}

A maxitive measure is \textit{$\mathrsfs{Q} \mathrsfs{F}$-smooth} if it is $\mathrsfs{Q}$-smooth and $\mathrsfs{F}$-smooth. 
The second part of the following result was proved by Puhalskii \cite[Theorem~1.7.8]{Puhalskii01} in the case where $L = \overline{\reels}_+$. 

\begin{proposition}\label{prop:f}
Assume that $L$ is a domain. 
Then, on a quasisober space, 
every $L$-valued tight weakly outer-continuous maxitive measure  
is $\mathrsfs{Q} \mathrsfs{F}$-smooth saturated. 
The converse statement holds in locally compact quasisober spaces and in completely metrizable spaces. 
\end{proposition}

\begin{proof}
Let $E$ be quasisober, let $\nu$ be an $L$-valued tight weakly outer-con\-tinuous maxitive measure on $\mathrsfs{B}$, and let $(F_j)_{j\in J}$ be a filtered family of closed subsets of $E$.  
Fix some compact Borel subset $K$, and let $F = \bigcap_{j \in J} F_j$. Then $F_j \cap K$ and $F \cap K$ are compact, hence $\uparrow\!\!(F_j \cap K)$ and $\uparrow\!\! (F \cap K)$ are compact saturated. 
Let us show that 
\begin{equation}\label{eq:inter}
\bigcap_{j\in J} \uparrow\!\!(F_j \cap K) = \uparrow\!\! (F \cap K). 
\end{equation} 
The inclusion $\supset$ is clear. For the reverse inclusion, let $x \in E$ such that $x \notin \uparrow\!\! (F \cap K)$. Then there is some open subset $G$ containing $F \cap K$ such that $x \notin G$. As a consequence, the compact subset $K$ is included in the union of the directed family $(G \cup (E \setminus F_j))_{j\in J}$, so there exists some $j_0 \in J$ such that $K \subset G \cup (E \setminus F_{j_0})$. This rewrites as $F_{j_0} \cap K \subset G$, so that $\uparrow\!\! (F_{j_0} \cap K) \subset G$. Hence, $x \notin \uparrow\!\! (F_{j_0} \cap K)$, and Equation~(\ref{eq:inter}) is proved.  

By Proposition~\ref{prop:k}, $\nu$ is $\mathrsfs{Q}$-smooth, so 
$$
\bigwedge_{j \in J} \nu(\uparrow\!\! (F_j \cap K)) = \nu(\uparrow\!\! (F \cap K)). 
$$ 
Since $\nu$ is weakly outer-continuous, $\nu$ is saturated, hence $\bigwedge_j \nu(F_j \cap K) = \nu(F \cap K)$. Now pick some lower-bound $\ell$ of the set $\{ \nu(F_j) : j \in J \}$. Thanks to Lemma~\ref{lem:jcont} (join-continuity of $L$), we have $\ell \leqslant \bigwedge_j (\nu(F_j \cap K) \oplus \nu(E \setminus K)) = \nu(F \cap K) \oplus \nu(E \setminus K)$. 
The tightness of $\nu$ and the join-continuity of $L$ imply $\ell \leqslant \nu(F)$, and the result is proved. 

For the converse statement, first assume that $E$ is locally compact quasisober, and let $\nu$ be an $L$-valued $\mathrsfs{Q} \mathrsfs{F}$-smooth saturated maxitive measure on $\mathrsfs{B}$. Then $\nu$ is weakly outer-continuous by Proposition~\ref{prop:k}. 
Moreover, the collection $\{ E \setminus K^{o} : K \in \mathrsfs{K} \}$ has empty intersection since $E$ is locally compact, is filtered, and is made of closed subsets. 
Since $\nu$ is $\mathrsfs{F}$-smooth, this implies $\bigwedge_{K \in \mathrsfs{K}} \nu(E \setminus K^{o}) = 0$. 
If $t \gg 0$, this gives some $K \in \mathrsfs{K}$ with $t \geqslant \nu(E \setminus K^{o})$, so that $t \geqslant \nu(E \setminus K)$. 
Since $L$ is continuous, we conclude that $\nu$ is tight. 

Now if $E$ is a completely metrizable space, the second part of the proof of Proposition~\ref{prop:k} still applies to show that an $L$-valued $\mathrsfs{F}$-smooth maxitive measure $\nu$ is weakly outer-continuous, for 
every compact subset $K$ is the filtered intersection of some family $(F_j)_{j}$ of closed subsets with  $K \subset F_j^{o}$. To see why this holds, define $F_j = \bigcap_{k=1}^j \overline{G_k}$ where, for all $k \geqslant 1$, $G_k$ is a finite union of open balls of radius $1/k$ covering $K$. 
For tightness, one can follow Puhalskii's proof \cite[Theorem~1.7.8]{Puhalskii01} (although this author considered only $\overline{\reels}_+$-valued maxitive measures). 
\end{proof}

\begin{problem}
Completely metrizable spaces and locally compact quasisober spaces are Baire spaces (see \cite[Theorem~3.47]{Aliprantis06} and \cite[Corollary~I-3.40.9]{Gierz03}). Does the previous result hold for Baire spaces?
\end{problem}

\begin{proposition}\label{prop:trpolish}
Assume that $L$ is a domain. 
Then, on a Polish space, every $L$-valued $\mathrsfs{F}$-smooth $\sigma$-maxitive measure is tight regular. 
\end{proposition}

\begin{proof}
Let $E$ be Polish, and let $\nu$ be an $L$-valued $\mathrsfs{F}$-smooth $\sigma$-maxitive measure on $\mathrsfs{B}$. 
Since $E$ is separable metrizable, every open subset is Lindel\"{o}f, hence the restriction of $\nu$ to $\mathrsfs{G}$ satisfies Equation~(\ref{eq:eqo}), i.e.\ $\nu$ is weakly inner-continuous. Now the result follows from Proposition~\ref{prop:f}. 
\end{proof}

\begin{remark}
For the case $L = \overline{\reels}_+$, one could prove Proposition~\ref{prop:trpolish} with the help of the Choquet capacitability theorem (see e.g.\ Molchanov \cite[Theorem~E.9]{Molchanov05} or Aliprantis and Border \cite[Theorem~12.40]{Aliprantis06}). 
\end{remark}

\begin{corollary}\label{coro:sigcomp}
Assume that $L$ is a domain. 
Then, on a $\sigma$-compact and metrizable space, every $L$-valued $\mathrsfs{K}$-smooth $\sigma$-maxitive measure is regular. 
\end{corollary}

\begin{proof}
Let $E$ be $\sigma$-compact and metrizable, and let $\nu$ be an $L$-valued $\mathrsfs{K}$-smooth $\sigma$-maxitive measure. 
Since $E$ is $\sigma$-compact, there is a sequence $(K_n)_{n}$ of compact subsets such that $E = \bigcup_n K_n$. 
Each of these $K_n$ is then a Polish space because $E$ is metrizable. 
By Proposition~\ref{prop:trpolish}, this implies that the restriction $\nu_n$ of $\nu$ to the Borel $\sigma$-algebra of $K_n$ is (tight) regular, hence completely maxitive by Proposition~\ref{lem:loccomp}. 
As a consequence, if $B \in \mathrsfs{B}$, then $\nu(B) = \bigoplus_{n} \nu(B \cap K_n) = \bigoplus_{n} \nu_n(B \cap K_n) = \bigoplus_{n} \bigoplus_{x \in B \cap K_n} \nu_n(\{x\}) = \bigoplus_{n} \bigoplus_{x \in B \cap K_n} \nu(\{x\}) = \bigoplus_{x\in B} \nu(\{x\})$, so $\nu$ is completely maxitive. 

Since complete maxitivity implies weak inner-continuity by Lemma~\ref{lem:wic}, it suffices to prove that $\nu$ is weakly outer-continuous in order to conclude that $\nu$ is regular. By Lemma~\ref{lem:reg0}, we only need to show that $\nu(\{x\}) = \nu^{+}(\{x\})$ for all $x$. 
So let $s \gg \nu(\{x\})$. Then $G := \{ y \in E : s \gg \nu(\{y \}) \}$ contains $x$. We prove that $G$ is open, i.e.\ that $F = E \setminus G$ is closed. 
Let $(y_n)$ be a sequence in $F$ with $y_n \rightarrow y$. If $Q_n$ is the topological closure of $\{ y_{n'} : n' \geqslant n \}$, then $Q_n$ is compact, and $\bigcap_n Q_n = \{ y\}$ since $E$ is Hausdorff. 
Since $\nu$ is $\mathrsfs{Q}$-smooth (i.e.\ $\mathrsfs{K}$-smooth), this gives $\bigwedge_n \nu(Q_n) = \nu(\{y\})$. If $y \notin F$, then $s \gg \bigwedge_n \nu(Q_n)$, hence there is some $n_0$ such that $s \gg \nu(Q_{n_0})$. Therefore, $s \gg \nu(\{y_{n_0}\})$, i.e.\ $y_{n_0} \notin F$, a contradiction. Thus, $y \in F$. Since $E$ is metrizable, it is first-countable, so this proves that $F$ is closed. 
So $G$ is open, contains $x$, and $s \geqslant \nu(G)$ because $\nu$ is completely maxitive. We deduce that $s \geqslant \nu^{+}(\{x\})$ and, with the continuity of $L$, that $\nu(\{x\}) \geqslant \nu^{+}(\{x\})$. From Lemma~\ref{lem:reg0} we conclude that $\nu$ is weakly outer-continuous, hence regular. 
\end{proof}

\begin{remark}
Part of the preceding corollary was proved by Miranda et al.\ \cite[Proposition~2.6]{Miranda04} in the case where $L  = \overline{\reels}_+$. 
It also uses ideas from \cite[Lemma~1.7.4]{Puhalskii01}. 
\end{remark}

\subsection{Cardinal densities of maxitive measures}

In this section we prove new results giving equivalent conditions for a maxitive measure $\nu$ on $\mathrsfs{B}$ to have a \textit{cardinal density}, that is a map $c : E  \rightarrow L$ such that 
$$
\nu(B) = \bigoplus_{x \in B} c(x), 
$$
for all $B \in \mathrsfs{B}$. 
As a special case, consider e.g.\ a finite set $E$ with the discrete topology. Then $\nu$ admits a cardinal density defined by $c(x) = \nu(\{ x \})$, since $B = \bigcup_{x\in B} \{ x \}$, where the union runs over a finite set.  In the general case, this reasoning may fail, for we may have $\nu(\{ x \}) = 0$ for all $x \in E$, even with a nonzero $\nu$, but it is tempting to consider $c^{+}(x) := \nu^{+}([x])$ instead, where $\nu^{+}$ is defined in Example~\ref{ex:nuplus}.  
This idea, which appeared in \cite{Heckmann98b, Heckmann98} and \cite{Akian99}, is effective and leads to Theorem~\ref{thm:reg0}. 

A map $c : E \rightarrow L$ is \textit{upper-semicontinuous} (or \textit{usc} for short) if, for all $t \in L$, the subset $\{ t \gg c \}$ is open. We refer the reader to Penot and Th\'era \cite{Penot82}, Beer \cite{Beer87}, van Gool \cite{vanGool92}, Gerritse \cite{Gerritse97}, Akian and Singer \cite{Akian03} for a wide treatment of upper-semicontinuity of poset-valued and domain-valued maps. 
Note that, if $L$ is a filtered-complete poset and $\nu$ is an $L$-valued maxitive map on $\mathrsfs{B}$, then the map $c^{+}$ defined by $c^{+}(x) = \nu^{+}([x])$ is usc. 

A map $c : E \rightarrow L$ is \textit{upper compact} if, for every $t \gg 0$, $\{ t \not\gg c \}$ is a compact subset of $E$. 

\begin{proposition}\label{prop:tensioneq}
Assume that $L$ is a domain, 
and let $\nu$ be an $L$-valued maxitive measure on $\mathrsfs{B}$. 
If $\nu$ is tight and outer-continuous, then $c^{+} : x \mapsto \nu^{+}([x])$ 
is upper compact. 
Conversely, if $\nu$ is weakly inner-continuous and $c^{+}$ is upper compact, then $\nu$ is tight. 
\end{proposition}

\begin{proof}
Assume that $\nu$ is tight and outer-continuous, and let $t \gg 0$. Since $\{ \nu(E \setminus K) : K \in \mathrsfs{K} \}$ is filtered with an infimum equal to $0$, the interpolation property implies that there is some $K \in \mathrsfs{K}$ such that $t \gg \nu(E \setminus K)$. Since $\nu$ is outer-continuous, we obtain $t \gg \bigoplus_{x \notin K} c^{+}(x)$. This shows that $\{ t \not\gg c^{+} \}$ is a subset of $K$. Since $c^{+}$ is usc, $\{ t \not\gg c^{+} \}$ is also closed, hence compact. 

Conversely, assume that $\nu$ is weakly inner-continuous and that $c^{+}$ is upper compact. Let $K_t$ denote the compact closed subset $\{ t \not\gg c^{+} \}$. Then 
$$
\bigwedge_{K \in\mathrsfs{K}} \nu(E \setminus K) \leqslant \bigwedge_{t \gg 0} \nu(E \setminus K_t). 
$$ 
Since $\nu$ is weakly inner-continuous and $G_t =E \setminus K_t$ is open for all $t \gg 0$, we have by Lemma~\ref{lem:reg0}
$$
\nu(G_t) = \bigoplus_{K \in \mathrsfs{K}, K \subset G_t} \nu^{+}(K) = \bigoplus_{K \in \mathrsfs{K}, K \subset G_t} \bigoplus_{x \in K} c^{+}(x), 
$$
thus $\nu(G_t) = \bigoplus_{x \in G_t} c^{+}(x)$, 
so that 
$$
\bigwedge_{K \in\mathrsfs{K}} \nu(E \setminus K) \leqslant \bigwedge_{t \gg 0} \bigoplus_{x \in E, t \gg c^{+}(x)} c^{+}(x) \leqslant \bigwedge_{t \gg 0} t = 0, 
$$ 
so $\nu$ is tight. 
\end{proof}

The following theorem summarizes many of the above results and highlights the relation between the existence of a density, regularity, and complete maxitivity. 
Part of it is due to \cite[Proposition~3.15]{Akian99} and \cite[Theorem~3.1]{Poncet10}. 
See also Norberg \cite{Norberg86}, Vervaat \cite{Vervaat97}. 
We also refer the reader to O'Brien and Watson \cite[Claim~2]{OBrien98} and Miranda et al.\ \cite[Proposition~2.3, Theorem~2.4]{Miranda04} for the case $L = \overline{\reels}_+$ and the link with the Choquet capacitability theorem. 

\begin{theorem}\label{thm:reg0}
Assume that $L$ is a domain and $E$ is a quasisober space. 
Let $\nu$ be an $L$-valued maxitive measure on $\mathrsfs{B}$. 
Then $\nu$ has a cardinal density if and only if $\nu$ is completely maxitive. 
Also, consider the following assertions: 
\begin{enumerate}
	\item\label{reg6} $\nu$ is regular, 
	\item\label{reg5} $\nu$ has a usc cardinal density, 
	\item\label{reg6bis} $\nu$ is outer-continuous and completely maxitive, 
	\item\label{reg6ter} $\nu$ is weakly outer-continuous and weakly inner-continuous, 
	\item\label{reg6qui} $\nu$ is weakly outer-continuous and $\sigma$-maxitive, 
	\item\label{reg6qui2} $\nu$ is weakly outer-continuous, 
	\item\label{reg6six} $\nu$ is $\mathrsfs{Q}$-smooth and saturated, 
	\item\label{reg6sep} $\nu$ is $\mathrsfs{Q}$-smooth, weakly inner-continuous, and saturated, 
	\item\label{reg6eig} $\nu$ is $\mathrsfs{Q}$-smooth, $\sigma$-maxitive, and saturated. 
\end{enumerate} 
Then (\ref{reg6}) $\Leftrightarrow$ (\ref{reg5}) $\Leftrightarrow$ (\ref{reg6bis}) $\Leftrightarrow$ (\ref{reg6ter}) $\Rightarrow$ (\ref{reg6qui}) $\Rightarrow$ (\ref{reg6qui2}) $\Rightarrow$ (\ref{reg6six}) $\Leftarrow$ (\ref{reg6sep}). 
Moreover, 
\begin{itemize}
	\item if $E$ is second-countable, then (\ref{reg6}) $\Leftrightarrow$ (\ref{reg5}) $\Leftrightarrow$ (\ref{reg6bis}) $\Leftrightarrow$ (\ref{reg6ter}) $\Leftrightarrow$ (\ref{reg6qui}); 
	\item if $E$ is locally compact, then (\ref{reg6sep}) $\Leftrightarrow$ (\ref{reg6}) $\Leftrightarrow$ (\ref{reg5}) $\Leftrightarrow$ (\ref{reg6bis}) $\Leftrightarrow$ (\ref{reg6ter}) and (\ref{reg6qui2}) $\Leftrightarrow$ (\ref{reg6six}); 
	\item if $E$ is $\sigma$-compact and metrizable, then (\ref{reg6eig}) $\Leftrightarrow$ (\ref{reg6}) $\Leftrightarrow$ (\ref{reg5}) $\Leftrightarrow$ (\ref{reg6bis}) $\Leftrightarrow$ (\ref{reg6ter}) $\Leftrightarrow$ (\ref{reg6qui}); 
	\item if $E$ is locally compact Polish, then (\ref{reg6sep}) $\Leftrightarrow$ (\ref{reg6eig}) $\Leftrightarrow$ (\ref{reg6}) $\Leftrightarrow$ (\ref{reg5}) $\Leftrightarrow$ (\ref{reg6bis}) $\Leftrightarrow$ (\ref{reg6ter}) $\Leftrightarrow$ (\ref{reg6qui})  and (\ref{reg6qui2}) $\Leftrightarrow$ (\ref{reg6six}). 
\end{itemize}
\end{theorem}

\begin{proof}
If $\nu$ is completely maxitive, then $\nu(B) = \nu(\bigcup_{x \in B} [x]) = \bigoplus_{x \in B} c(x)$ by Lemma~\ref{lem:tilde}, where $c(x) = \nu([x])$, hence $\nu$ has a cardinal density. The reverse assertion is straightforward. 

(\ref{reg6}) $\Rightarrow$ (\ref{reg5}) Assume that $\nu$ is regular. 
Then $\nu(\uparrow\!\! K) = \bigoplus_{x \in K} c^{+}(x)$ for all $K \in \mathrsfs{K}$ by Lemma~\ref{lem:reg0}, where $c^{+}(x) = \nu^{+}([x])$. 
By inner-continuity of $\nu$, $\nu(B) = \bigoplus_{K \in \mathrsfs{K}, K \subset B} \nu(\uparrow\!\! K) = \bigoplus_{K \in \mathrsfs{K}, K \subset B} \bigoplus_{x \in K} c^{+}(x) = \bigoplus_{x \in B} c^{+}(x)$, for all Borel subsets $B$, i.e.\ $\nu$ has a usc cardinal density. 

(\ref{reg5}) $\Rightarrow$ (\ref{reg6}) Assume that $\nu$ has a usc cardinal density $c$. 
Then $\nu$ is weakly inner-continuous. 
Let us show that, if $K \in \mathrsfs{K}$, then $\nu(K) = \nu^{+}(K)$. So let $u \gg \nu(K)$. Since $\nu(K) = \bigoplus_{x \in K} c(x)$, we have $K \subset G$ where $G = \{ u \gg c \}$ is open. Moreover, $\nu(G) = \bigoplus_{x \in G} c(x) \leqslant u$. Therefore, $\nu(K) = \nu^{+}(K)$ by continuity of $L$. This implies that $\nu$ is regular by Proposition~\ref{lem:loccomp}. 

So now the implications (\ref{reg5}) $\Rightarrow$ (\ref{reg6bis}) $\Rightarrow$ (\ref{reg6ter}) $\Rightarrow$ (\ref{reg6}) are clear (use Lemma~\ref{lem:wic} and Proposition~\ref{lem:loccomp}). 
Using Proposition~\ref{prop:k}, it is also straightforward that (\ref{reg6bis}) $\Rightarrow$ (\ref{reg6qui}) $\Rightarrow$ (\ref{reg6qui2}) $\Rightarrow$ (\ref{reg6six}) $\Leftarrow$ (\ref{reg6sep}). 

If $E$ is second-countable, use Corollary~\ref{coro:sc}. 
If $E$ is locally compact, use Proposition~\ref{prop:k}. 
If $E$ is $\sigma$-compact and metrizable, use Corollary~\ref{coro:sigcomp}. 
If $E$ is locally compact Polish, use Proposition~\ref{prop:k} and Corollary~\ref{coro:sigcomp}. 
\end{proof}

\begin{corollary}
Assume that $L$ is a domain and $E$ is a quasisober space. 
If $\nu$ is a regular maxitive measure on $E$, then $c^{+}(x) = \nu([x])$ for all $x \in E$, and $c^{+}$ is the maximal (usc) cardinal density of $\nu$. 
\end{corollary}

\begin{theorem}\label{thm:reg0tight}
Assume that $L$ is a domain and $E$ is a quasisober space. 
Let $\nu$ be an $L$-valued maxitive measure on $\mathrsfs{B}$. 
Also, consider the following assertions: 
\begin{enumerate}
	\item\label{reg6t} $\nu$ is tight regular, 
	\item\label{reg5t} $\nu$ has an upper compact usc cardinal density, 
	\item\label{reg6qui2t} $\nu$ is tight weakly outer-continuous, 
	\item\label{reg6sixt} $\nu$ is $\mathrsfs{Q} \mathrsfs{F}$-smooth and saturated, 
	\item\label{reg6sept} $\nu$ is $\mathrsfs{Q} \mathrsfs{F}$-smooth, weakly inner-continuous, and saturated, 
	\item\label{reg6eigt} $\nu$ is $\mathrsfs{Q} \mathrsfs{F}$-smooth, $\sigma$-maxitive, and saturated, 
\end{enumerate} 
Then (\ref{reg6t}) $\Leftrightarrow$ (\ref{reg5t}) $\Rightarrow$ (\ref{reg6qui2t}) $\Rightarrow$ (\ref{reg6sixt}) $\Leftarrow$ (\ref{reg6sept}). 
Moreover, 
\begin{itemize}
	\item if $E$ is locally compact, then (\ref{reg6sept}) $\Leftrightarrow$ (\ref{reg6t}) $\Leftrightarrow$ (\ref{reg5t}) and (\ref{reg6qui2t}) $\Leftrightarrow$ (\ref{reg6sixt});  
	\item if $E$ is completely metrizable, then (\ref{reg6qui2t}) $\Leftrightarrow$ (\ref{reg6sixt}); 
	\item if $E$ is Polish, then (\ref{reg6eigt}) $\Leftrightarrow$ (\ref{reg6t}) $\Leftrightarrow$ (\ref{reg5t}); 
	\item if $E$ is locally compact Polish, then (\ref{reg6sept}) $\Leftrightarrow$ (\ref{reg6eigt}) $\Leftrightarrow$ (\ref{reg6t}) $\Leftrightarrow$ (\ref{reg5t}) and (\ref{reg6qui2t}) $\Leftrightarrow$ (\ref{reg6sixt}). 
\end{itemize}
\end{theorem}

\begin{proof}
The equivalence (\ref{reg6t}) $\Leftrightarrow$ (\ref{reg5t}) is a consequence of Proposition~\ref{prop:tensioneq}. 
For the implications (\ref{reg6t}) $\Rightarrow$ (\ref{reg6qui2t}) $\Rightarrow$ (\ref{reg6sixt}) $\Leftarrow$ (\ref{reg6sept}), use Proposition~\ref{prop:f}. 

If $E$ is completely metrizable or locally compact, use Proposition~\ref{prop:f}. 
If $E$ is Polish, use Proposition~\ref{prop:trpolish}. 
\end{proof}

\section{Regularity of optimal measures on metrizable spaces}\label{sec:metriz} 

Let $E$ be a topological space with Borel $\sigma$-algebra $\mathrsfs{B}$. 
An $L$-valued maxitive measure $\nu$ on $\mathrsfs{B}$ is \textit{continuous from above} 
if $\nu(B) = \bigwedge_n \nu(B_n)$, for all $B_1 \supset B_2 \supset \ldots \in \mathrsfs{B}$ such that $B = \bigcap_n B_n$. 
An \textit{optimal measure} is a continuous from above $\sigma$-maxitive measure. 

The following result generalizes the Murofushi--Sugeno--Agebko theorem (see \cite[Proposition~I-8.2]{Poncet11}). 

\begin{proposition}
Assume that $L$ is a domain. 
An $L$-valued maxitive measure on $\mathrsfs{B}$ is an optimal measure if and only if it is a continuous from above. 
\end{proposition}

\begin{proof}
Let $\nu$ be an $L$-valued continuous from above maxitive measure on $\mathrsfs{B}$, and let us show that $\nu$ is $\sigma$-maxitive. 
So let $B_1,  B_2, \ldots \in \mathrsfs{B}$ and $B = \bigcup_n B_n$, and let $C_n = B \setminus B'_n$ with $B'_n = B_1 \cup \cdots \cup B_n$. 
By continuity from above, we have $\bigwedge_n \nu(C_n) = 0$. 
Let $u$ be an upper bound of $\{ \nu(B_n) : n \geqslant 1 \}$, and let $s \gg u$. Then $s \geqslant \nu(C_{n_0})$ for some $n_0$, so $\nu(B) = \nu(B'_{n_0}) \oplus \nu(C_{n_0}) = \nu(B_1) \oplus \cdots \oplus \nu(B_{n_0}) \oplus \nu(C_{n_0}) \leqslant s$. Thus, $\nu(B) \leqslant u$, which shows that $\nu(B)$ is the least upper-bound of $\{ \nu(B_n) : n \geqslant 1 \}$. 
\end{proof}

Rie{\v{c}}anov{\'a} \cite{Riecanova84} studied the regularity of certain $S$-valued set functions, for some conditionally-complete ordered semigroup $S$ satisfying a series of conditions, among which the separation of points by continuous functionals. In the following lines we closely follow her approach, although we do not use directly her results, for our approach better matches the special case of $L$-valued optimal measures. In particular, $L$ is not assumed to be a semigroup, nor to be conditionally-complete. Unlike Rie{\v{c}}anov{\'a}, we do not examine the case of optimal measures defined on the collection of Baire (rather than Borel) subsets of a metrizable space, but we believe that this could be done with little additional effort. 

The following lemma is based on \cite[Proposition~IV-3.1]{Gierz03}. It allows one to generalize most theorems that hold for $[0, 1]$-valued maxitive measures to domain-valued maxitive measures. 

\begin{lemma}[Compare with Gierz et al.\ \protect{\cite[Proposition~IV-3.1]{Gierz03}}]\label{lem:gierz}
Assume that $L$ is a domain and let $s, t \in L$ such that $s \not\leqslant t$. 
Then there exists a map $\varphi : L \rightarrow [0, 1]$ preserving filtered infima and arbitrary existing suprema such that $\varphi(s) = 1$ and $\varphi(t) = 0$. 
\end{lemma}

\begin{proof}
We add a top to $L$, i.e.\ we let $\overline{L} = L \cup \{ \top \}$ with $r \leqslant \top$ for all $r \in L$. 
Then $\overline{L}$ is a domain, and by \cite[Proposition~IV-3.1]{Gierz03} there exists a map $\psi : \overline{L} \rightarrow [0, 1]$ preserving filtered infima and arbitrary existing suprema such that $\psi(s) = 1$ and $\psi(t) = 0$. It then suffices to define $\varphi = \psi|_L$. 
\end{proof}

\begin{proposition}\label{prop:metric}
Assume that $L$ is a domain. 
Then, on a metrizable space, every $L$-valued optimal measure $\nu$ satisfies  
$$
\nu(B) = \bigwedge_{G \in \mathrsfs{G}, G \supset B} \nu(G) = \bigoplus_{F \in \mathrsfs{F}, F \subset B} \nu(F), 
$$
for all $B \in \mathrsfs{B}$. 
\end{proposition}

\begin{proof}
Let $E$ be a metrizable space and $d$ be a metric generating the topology. 
Let $\varphi : L \rightarrow [0, 1]$ be a map preserving filtered infima and arbitrary existing suprema, and let $\nu_{\varphi}$ be the map defined on $\mathrsfs{B}$ by $\nu_{\varphi}(B) = \varphi(\nu(B))$. The properties of $\varphi$ imply that $\nu_{\varphi}$ is an optimal measure. 
Let $\mathrsfs{A}$ be the collection of all $B \in \mathrsfs{B}$ such that $\nu_{\varphi}(G \setminus F) \leqslant 1/2$, for some open subset $G$ and closed subset $F$ such that $G \supset B \supset F$. Let us show first that  $\mathrsfs{A}$ contains all open subsets, so let $B$ be open. Let $F_n = \{ x \in E : d(x, E\setminus B) \geqslant n^{-1} \}$. Then $(F_n)_{n \geqslant 1}$ is a nondecreasing family of closed subsets whose union is $B$. Since  $\nu_{\varphi}$ is an optimal measure, $\nu_{\varphi}(B \setminus F_n)$ tends to $0$ when $n \uparrow \infty$. Thus, we can find some closed subset $F \subset B$ with $\nu_{\varphi}(B \setminus F) \leqslant 1/2$, and this proves that $B \in \mathrsfs{A}$. 

We now show that $\mathrsfs{A}$ is a $\sigma$-algebra. Clearly, $\emptyset \in \mathrsfs{A}$, and $B \in \mathrsfs{A}$ implies $E \setminus B \in \mathrsfs{A}$. Let $(B_n)_{n \geqslant 1}$ be a family of elements of $\mathrsfs{A}$. We prove that $B = \bigcup_{n} B_n \in \mathrsfs{A}$. For all $n$, there are some $G_n \supset B_n \supset F_n$ satisfying $\nu_{\varphi}(G_n  \setminus F_n) \leqslant 1/2$. If $G = \bigcup_{n} G_n$ and $F = \bigcup_{n} F_n$, then $G \supset B \supset F$ and $\nu_{\varphi}(G \setminus F) \leqslant 1/2$. However, $F$ is not closed in general. So let $H_n$ denote the closed subset $\bigcup_{k = 1}^n F_k$. As above, $(H_n)_{n \geqslant 1}$ is a nondecreasing family of closed subsets whose union is $F$, so we can find some closed subset $H_{n_0} \subset F$ with $\nu_{\varphi}(F \setminus H_{n_0}) \leqslant 1/2$, hence $\nu_{\varphi}(G \setminus H_{n_0}) \leqslant 1/2$. 	Consequently, $\mathrsfs{A}$ coincides with the Borel $\sigma$-algebra $\mathrsfs{B}$. 

Assume that, for some $B \in \mathrsfs{B}$, $\nu^{+}(B)$ is not the least upper-bound of $\{ \nu(F) : F \in \mathrsfs{F}, F \subset B \}$. Hence, there exists some upper-bound $u \in L$ of $\{ \nu(F) : F \in \mathrsfs{F}, F \subset B \}$ such that $\nu^{+}(B) \not\leqslant u$. Since $L$ is a domain, there exists some $\varphi : L \rightarrow [0, 1]$ that preserves filtered infima and arbitrary existing suprema such that $\varphi(\nu^{+}(B)) = 1$ and $\varphi(u) = 0$ (see Lemma~\ref{lem:gierz}). 
The previous point gives the existence of some $G \supset B \supset F$ such that $\nu_{\varphi}(G \setminus F) \leqslant 1/2$. Moreover, $\varphi(\nu^{+}(B)) = 1$ implies $\nu_{\varphi}(G) = 1$, and $\varphi(u) = 0$ implies $\nu_{\varphi}(F) = 0$. But $1 = \nu_{\varphi}(G) = \nu_{\varphi}(G \setminus F) \oplus \nu_{\varphi}(F) \leqslant 1/2$, a contradiction. 
\end{proof}

\begin{corollary}\label{coro:sclc}
Assume that $L$ is a domain.  
Then, on a separable metrizable space, every $L$-valued optimal measure is regular. 
\end{corollary}

\begin{proof}
Let $E$ be a separable metrizable space, and let $\nu$ be an $L$-valued optimal measure on $\mathrsfs{B}$. 
Then $\nu$ is outer-continuous by Proposition~\ref{prop:metric}. As a separable metrizable space, $E$ is second-countable, so $\nu$ is also inner-continuous by Corollary~\ref{coro:sc}, hence regular. 
\end{proof}

\begin{remark}
The previous result was proved by Murofushi and Sugeno \cite[Theorem~4.1]{Murofushi93} for $\overline{\reels}_+$-valued optimal measures. 
\end{remark}

\begin{remark}
Recall that a topological space $E$ is separable metrizable in all the following cases: 
\begin{enumerate}
	\item if $E$ is second-countable regular Hausdorff; 
	\item if $E$ is $\sigma$-compact and metrizable; 	
	\item if $E$ is Polish (this results from the definition of a Polish space!). 
\end{enumerate}
\end{remark}


Part of the following result is included in Proposition~\ref{prop:trpolish}. 

\begin{proposition}
Assume that $L$ is a domain. 
Then, on a Polish space or on a $\sigma$-compact and metrizable space, every $L$-valued optimal measure is tight  regular. 
\end{proposition}

The proof is inspired by that of \cite[Theorem~1.7.8]{Puhalskii01}. 

\begin{proof}
We only have to prove tightness. 
First assume that $E$ is a Polish space and let $\nu$ be an $L$-valued optimal measure on $\mathrsfs{B}$. 
Since $E$ is separable, there is some sequence $(x_n)$ dense in $E$. Let $\epsilon \gg 0$. 
Let $F_{n,p} = B_{1,p} \cup \ldots \cup B_{n,p}$, where $B_{n,p}$ is the closed ball of radius $1/p$ and center $x_n$. Then, for all $p$, $E = \bigcup_{n} F_{n,p}$. Since $\nu$ is optimal, there is some $n_p$ such that $\epsilon \geqslant \nu(E \setminus F_{n_p, p})$. 
Let $K_{\epsilon}$ denote the subset $\bigcap_{p} F_{n_p, p}$. For all $\alpha > 0$, $K_{\epsilon}$ can be covered by a finite number of balls of radius at most $\alpha$, i.e.\ $K_{\epsilon}$ is totally bounded. Since $E$ is completely metrizable, $K_{\epsilon}$ is compact. Moreover, $\epsilon \geqslant \nu(E \setminus K_{\epsilon})$, for all $\epsilon \gg 0$. Thus, $\nu$ is tight. 

For the case where $E$ is $\sigma$-compact and metrizable, a similar proof can be given, for one can write $E = \bigcup_{n} F_{n,p}$, with $F_{n,p} = F_{n,1}$ compact. 
\end{proof}

\section{Decomposition of maxitive measures}\label{sec:decompose}

In \cite{Poncet10}, we developed part of the following material in a non-topolo\-gi\-cal framework. 
Here $E$ is again a quasisober topological space, and $\mathrsfs{B}$ denotes its collection of Borel subsets. A poset is a \textit{lattice} if every nonempty finite subset has a supremum and an infimum. A lattice is \textit{distributive} if finite infima distribute over finite suprema, and \textit{conditionally-complete} if every nonempty subset bounded above has a supremum. 
According to an assumption made all along this paper (see Section~\ref{sec:domain}), a continuous conditionally-complete lattice, which is a special case of domain, will always have a bottom element $0$. 

\begin{definition}\label{partreg}
Assume that $L$ is a continuous conditionally-complete lattice. Let $\nu$ be an $L$-valued maxitive measure on $\mathrsfs{B}$. Then the \textit{regular part} of $\nu$ is the map defined on $\mathrsfs{B}$ by 
$$
\lfloor \nu \rfloor(B) = \bigoplus_{K \in\mathrsfs{K}, K\subset B}  \nu^{+}(K). 
$$
\end{definition} 

The following proposition confirms that the terminology is appropriate. 

\begin{proposition}
Assume that $L$ is a continuous conditionally-complete lattice and $E$ is a quasisober space. Let $\nu$ be an $L$-valued maxitive measure on $\mathrsfs{B}$. Then the regular part of $\nu$ is a regular maxitive measure on $\mathrsfs{B}$, with density $c^{+} : x \mapsto  \nu^{+}([x])$.  
Moreover, 
$\lfloor\lfloor \nu \rfloor\rfloor = \lfloor \nu \rfloor$. 
\end{proposition}

\begin{proof}
By Lemma~\ref{lem:reg0}, $\nu^{+}(K) = \bigoplus_{x \in K} c^{+}(x)$ for all compact Borel subsets $K$ of $E$, so we have $\lfloor \nu \rfloor(B) = \bigoplus_{x \in B} c^{+}(x)$, for all $B \in \mathrsfs{B}$. This shows that $\lfloor \nu \rfloor$ has a usc cardinal density, hence is regular by Theorem~\ref{thm:reg0}. 
Outer-continuity of $\lfloor \nu \rfloor$ implies that $\lfloor \nu \rfloor^{+}(K) = \lfloor \nu \rfloor(K) = \nu^{+}(K)$, for all $K \in \mathrsfs{K}$, so $\lfloor\lfloor \nu \rfloor\rfloor = \lfloor \nu \rfloor$. 
\end{proof}

The following theorem states the existence of a \textit{singular part} $\bot\nu$ of a maxitive measure $\nu$. 

\begin{theorem}\label{thm:singpart}
Assume that $L$ is a continuous conditionally-complete distributive lattice and $E$ is a quasisober space. 
Let $\nu$ be an $L$-valued maxitive measure on $\mathrsfs{B}$. Then there exists a least maxitive measure $\bot\nu$ on $\mathrsfs{B}$, called the \textit{singular part} of $\nu$, such that the decomposition 
\begin{equation}\label{eq:decomp}
\nu^{+} = \lfloor \nu \rfloor \oplus \bot\nu 
\end{equation}
holds. Moreover, the singular part $\bot\nu$ vanishes on all subsets of the form $[x]$, $x\in E$, and the singular part of the regular part of $\nu$ equals $0$, i.e.\ $\bot \lfloor \nu \rfloor = 0$. 
\end{theorem}

\begin{proof}
We give a constructive proof for the existence of $\bot\nu$. 
Let $\bot\nu(B) = \bigwedge\{ t \in L : B \in \mathrsfs{I}_t \}$, where 
$$
\mathrsfs{I}_t := \{ B \in \mathrsfs{B} : \forall A \in \mathrsfs{B}, A \subset B \Rightarrow \nu^{+}(A) \leqslant \lfloor \nu \rfloor(A) \oplus t\}. 
$$ 
Then $(\mathrsfs{I}_t)_{t \in L}$ is a nondecreasing family of ideals of $\mathrsfs{B}$, and distributivity of $L$ implies that $\{ t \in L : B \in \mathrsfs{I}_t \}$ is a filter,  
for every $B \in \mathrsfs{B}$. From \cite[Proposition~2.3]{Poncet10}, we deduce that $\bot\nu$ is a maxitive measure.  

Since $B \in \mathrsfs{I}_t$ for $t = \nu^{+}(B)$, we have $\nu^{+}(B) \geqslant \bot\nu(B)$, thus $\nu^{+} \geqslant \lfloor \nu \rfloor \oplus \bot\nu$. 
For the reverse inequality, one may use the fact that continuity implies join-continuity (see Lemma~\ref{lem:jcont}). 
The fact that $\bot\nu$ is the \textit{least} maxitive measure satisfying Equation~(\ref{eq:decomp}) is straightforward. 

The fact that $\bot\nu([x]) = 0$ and the identity $\bot\lfloor \nu \rfloor = 0$ follow from the definition of the singular part (and, for the latter property, from the fact that $\lfloor\lfloor \nu \rfloor\rfloor = \lfloor \nu \rfloor$). 
\end{proof}

As a consequence of the previous result we have the following corollaries. The proof of the first of them is clear. 

\begin{corollary}\label{coro:reg}
Under the conditions of Theorem~\ref{thm:singpart}, the following are equivalent if $\nu$ is outer-continuous: 
\begin{enumerate}
	\item $\nu$ is the regular part of some $L$-valued maxitive measure, 
	\item the singular part of $\nu$ is identically $0$, 
	\item $\nu$ is regular.  
\end{enumerate}
\end{corollary}

\begin{corollary}\label{coro:res}
Under the conditions of Theorem~\ref{thm:singpart}, the following are equivalent if $\nu$ is outer-continuous: 
\begin{enumerate}
	\item\label{res1} $\nu$ is the singular part of some $L$-valued maxitive measure, 
	\item\label{res2} the regular part of $\nu$ is identically $0$, 
	\item\label{res3} $\nu$ satisfies $\nu(K) = 0$ for all $K \in \mathrsfs{K}$. 
\end{enumerate}
\end{corollary}

\begin{proof}
It is straightforward that (\ref{res3}) $\Leftrightarrow$ (\ref{res2}) $\Rightarrow$ (\ref{res1}). Let us show that (\ref{res1}) $\Rightarrow$ (\ref{res2}), so assume that $\nu = \bot\tau$, for some $L$-valued maxitive measure $\tau$. Note that $\bot(\bot\tau) \geqslant \bot\tau$, for 
$\lfloor \bot\tau \rfloor$ is regular and less than $\tau^{+}$, hence is less than $\lfloor \tau \rfloor$. Thus, $\tau^{+} = \lfloor \tau \rfloor \oplus \bot\tau \leqslant \lfloor \tau \rfloor \oplus (\bot\tau)^{+} =  \lfloor \tau \rfloor \oplus \lfloor \bot\tau \rfloor \oplus \bot(\bot\tau) = \lfloor \tau \rfloor \oplus \bot(\bot\tau) \leqslant \tau^{+}$. This gives $\tau^{+} = \lfloor \tau \rfloor \oplus \bot(\bot\tau)$, hence $\bot(\bot\tau) \geqslant \bot\tau$. 
Now $\nu = \nu^{+} \geqslant \bot\nu = \bot(\bot\tau) \geqslant \bot\tau = \nu$, so that $\nu = \bot\nu$. 
Since $\nu$ is outer-continuous, $\nu^{+}([x]) = \nu([x]) = \bot\nu([x]) = 0$ for all $x \in E$, so  
$\lfloor \nu \rfloor (B) = \bigoplus_{x \in B} \nu^{+}([x]) = 0$, for all $B \in \mathrsfs{B}$. 
\end{proof}


\begin{corollary}
Assume that $L$ is a continuous conditionally-complete distributive lattice. 
Then, on a metrizable space, 
every $L$-valued optimal measure $\nu$ can be decomposed as 
\begin{equation}\label{eq:decomp2}
\nu = \lfloor \nu \rfloor \oplus \bot\nu,  
\end{equation}
where $\lfloor \nu \rfloor$ is a regular optimal measure and $\bot\nu$ is an optimal measure satisfying $\bot\nu(K) = 0$ for all $K \in \mathrsfs{K}$. 
\end{corollary}

\begin{proof}
By Proposition~\ref{prop:metric}, the optimal measure $\nu$ is outer-continuous, i.e.\ $\nu = \nu^{+}$. 
Let $\tau = \bot \nu$. Since $\tau \leqslant \nu^{+}  = \nu$, it is easily seen that the maxitive measure $\tau$ is optimal, hence outer-continuous. Applying Corollary~\ref{coro:res} to $\tau$, we deduce that $\tau$, as the singular part of $\nu$, satisfies $\tau(K) = 0$ for all $K \in \mathrsfs{K}$. 
\end{proof}

\section{Conclusion and perspectives}\label{sec:con}

It would be interesting to reformulate the results of this work in terms of Baire subsets rather than Borel subsets.  

\begin{acknowledgements}
I would like to gratefully thank Prof.\ Jimmie D.\ Lawson for his comments and valuable suggestions; they enabled me to correctly capture the  non-Hausdorff setting. I am also indebted in an anonymous referee whose remarks greatly improved the manuscript. 
\end{acknowledgements}

\bibliographystyle{plain}

\def\cprime{$'$} \def\cprime{$'$} \def\cprime{$'$} \def\cprime{$'$}
  \def\ocirc#1{\ifmmode\setbox0=\hbox{$#1$}\dimen0=\ht0 \advance\dimen0
  by1pt\rlap{\hbox to\wd0{\hss\raise\dimen0
  \hbox{\hskip.2em$\scriptscriptstyle\circ$}\hss}}#1\else {\accent"17 #1}\fi}
  \def\ocirc#1{\ifmmode\setbox0=\hbox{$#1$}\dimen0=\ht0 \advance\dimen0
  by1pt\rlap{\hbox to\wd0{\hss\raise\dimen0
  \hbox{\hskip.2em$\scriptscriptstyle\circ$}\hss}}#1\else {\accent"17 #1}\fi}

\end{document}